\documentclass[12pt]{amsart}
\usepackage{amscd}
\usepackage{amssymb}
\newcommand{\be}{\begin{equation}}
\newcommand{\ee}{\end{equation}}
\newcommand{\ba}{\begin{eqnarray}}
\newcommand{\ea}{\end{eqnarray}}
\newcommand{\ban}{\begin{eqnarray*}}
	\newcommand{\ean}{\end{eqnarray*}}

\def\XXint#1#2#3{{\setbox0=\hbox{$#1{#2#3}{\int}$}
		\vcenter{\hbox{$#2#3$}}\kern-.5\wd0}}

\newcommand{\Rk}{\noindent {\bf Remark} }

\newtheorem{theo}{Theorem}[section]

\begin{document}
	\newtheorem{defn}[theo]{Definition}
	\newtheorem{ques}[theo]{Question}
	\newtheorem{lem}[theo]{Lemma}
	\newtheorem{prop}[theo]{Proposition}
	\newtheorem{coro}[theo]{Corollary}
	\newtheorem{ex}[theo]{Example}
	\newtheorem{note}[theo]{Note}
	\newtheorem{conj}[theo]{Conjecture}
	\makeatletter
	\@namedef{subjclassname@2020}{%
		\textup{2020} Mathematics Subject Classification}
	\makeatother

	\title[Ricci Shrinker with Constant scalar curvature]{A note on Rigidity of  Shrinking Gradient Ricci Solitons with Constant Scalar Curvature}
	\author{Chen Wang}
	\address[Chen Wang]
	{School of Science, Zhejiang Sci-Tech University, Hangzhou 310018, China}
	\email{2023210103029@mails.zstu.edu.cn}

	\author{Guoqiang Wu}
	\address[Guoqiang Wu]
	{School of Science, Zhejiang Sci-Tech University, Hangzhou 310018, China}
	\email{gqwu@zstu.edu.cn}

	
	\date{}
	\maketitle

	\begin{abstract} Let $(M^n, g, f)$ be an $n$-dimensional complete noncompact gradient shrinking Ricci soliton with the equation $Ric+\nabla^2f= \frac{1}{2}g$.

 1. If its scalar curvature is $\frac{k}{2}$, Ricci curvature is nonnegative and sectional curvature has upper bound $\frac{1}{2(k-1)}$, we prove that the Ricci shrinker is isometric to a finite quotient of $\mathbb{R}^{n-k}\times \mathbb{S}^k$.

 2. If $M$ has constant scalar curvature $R=\frac{n-2}{2}$, and  each level set of $f$ has vanishing Weyl curvature, we prove that it is a finite quotient of $\mathbb{R}^2\times \mathbb{S}^{n-2}$. This can be seen a generalization of Cheng-Zhou's four dimensional  result  \cite{Cheng-Zhou} to high dimension, since the level set of the potential function $f$ has vanishing Weyl curvature automatically when $n=4$.
	\end{abstract}
	
	\section{Introduction}
	Let  $(M^n, g, f)$ be an $n$-dimensional complete gradient Ricci soliton with the potential function $f$ satisfying
	\begin{align}\label{soliton}
	\text{Ric}+\nabla^2f=\lambda g
	\end{align}
	for some constant $\lambda$, where $\text{Ric}$ is the Ricci tensor of $g$ and $\nabla^2f$ denotes the Hessian of the potential function $f$.
	The Ricci soliton is said to be shrinking, steady, or expanding accordingly as $\lambda$ is positive, zero, or negative, respectively. By rescalling the metric $g$ by a positive constant, we can assume $\lambda\in \{\frac{1}{2}, 0, -\frac{1}{2}\}$.
	\smallskip	
	
	A gradient Ricci soliton is a self-similar solution to the Ricci flow which flows by diffeomorphism and homothety. The study of solitons has become increasingly important in both the study of the Ricci flow introduced by Hamilton \cite{Hamilton} and metric measure theory. Solitons play a direct role as singularity dilations in the Ricci flow proof of uniformization.  Due to the work of Perelman \cite{Perelman2}, Ni-Wallach \cite{Ni-Wallach}, Cao-Chen-Zhu \cite{Cao-Chen-Zhu}, the classification of three dimensional shrinking gradient Ricci soliton is complete. For more work on the classification of gradient Ricci soliton under various curvature condition, see \cite{Brendle1, Brendle2, Cao-Chen, Cao-Chen2, Cao-Chen-Zhu, Cao-Xie, Cao-Wang-Zhang, Chen-Wang,Kotschwar, Eminenti-LaNave-Mantegazza,Munteanu-Wang4,  Naber, Petersen-Wylie, Pigola-Rimoldi-Setti, Wu-Zhang, Wu-Wu-Wylie,  Zhang}.

	\smallskip	
	In this paper, we focus our attention on $n$-dimensional gradient shrinking Ricci solitons with constant scalar curvature. Recall that in Petersen and Wylie's paper \cite{Petersen-Wylie}, a gradient Ricci soliton $(M, g)$ is said to be rigid if it is isometric to a quotient $\mathbb{N} \times \mathbb{R}^k$, the product soliton of an Einstein manifold $\mathbb{N}$ of positive scalar curvature with the Gaussian soliton $\mathbb{R}^k$.
	Conversely, for the complete shrinking case, Prof. Huai-Dong Cao raised the following
	
	\smallskip	
	\noindent {\bf Conjecture}:
	Let $(M^n, g, f)$, $n\geq 4$, be a complete $n$-dimensional gradient shrinking Ricci soliton. If $(M, g)$ has constant scalar curvature, then it must be rigid,
	i.e., a finite quotient of $\mathbb{N}^k\times \mathbb{R}^{n-k}$ for some Einstein manifold $\mathbb{N}$ of positive scalar curvature.
	
	\smallskip	
	Petersen and Wylie \cite{Petersen-Wylie} proved that a complete gradient Ricci soliton is rigid if and only if it has
	constant scalar curvature and is radially flat, that is, the sectional curvature $K(\cdot, \nabla f)=0$. Fern\'{a}ndez-L\'{o}pez and Garc\'\i a-R\'\i o \cite{FR16} obtained that the soliton is rigid if and only if the Ricci curvature has constant rank. They also derived the following results for complete $n$-dimensional gradient Ricci solitons \eqref{soliton} with constant scalar curvature $R$: (i) The possible value of $R$ is $\{0, \lambda, \cdots, (n-1)\lambda, n\lambda\}$. (ii) If $R$ takes the value $(n-1)\lambda$,  then the soliton must be rigid. (iii) In the shrinking case, there is  no any complete gradient shrinking Ricci soliton  with $R=\lambda$.
	(iv) Any $n$-dimensional gradient shrinking Ricci soliton with constant scalar curvature $R=(n-2)\lambda$ has non-negative Ricci curvature.
	
	\smallskip	
	Several years ago, Cheng and Zhou \cite{Cheng-Zhou} confirmed Cao's conjecture in dimension $n=4$.
Very recently, the authors gave a simple proof of Cheng-Zhou's result in \cite{Ou-Qu-Wu}. Later, the authors \cite{Li-Ou-Qu-Wu} finished the $5$-dimensional case when the scalar curvature is $\frac{3}{2}$. In this paper, we want to generalize this results to high dimension in some sense.

	Our main theorem is as follows.
	\begin{theo}\label{main}
		Suppose $(M^n, g, f)$ is an $n$-dimensional shrinking gradient Ricci soliton with $R=\frac{n-2}{2}$, if each level set of $f$ has vanishing Weyl curvature,  then it is isometric to a finite quotient of $\mathbb{R}^2 \times \mathbb{S}^{n-2}$.
	\end{theo}
\Rk Theorem \ref{main} can be seen a generalization of Cheng-Zhou's result \cite{Cheng-Zhou} in dimension four, since the level set of the potential function $f$ has vanishing Weyl curvature automatically when $n=4$.

	When the sectional curvature has sharp upper bounded, we also derive the following result.
\begin{theo}\label{theorem2}
Let $(M^n, g, f)$ be an $n$-dimensional shrinking gradient Ricci soliton with $R=\frac{k}{2}$. If it has nonnegative Ricci curvature and its sectional section has upper bound $\frac{1}{2(k-1)}$, then it is isometric to a finite quotient of  $\mathbb{S}^k\times \mathbb{R}^{n-k}$.
\end{theo}
 	
\Rk The above bound $\frac{1}{2(k-1)}$ is sharp, because $S^k$ has constant sectional curvature $\frac{1}{2(k-1)}$.
	
	%

	\smallskip

	The paper is organized as follows.
	In Section \ref{sec2}, we recall the notations and basic formulas on gradient shrinking Ricci solitons with constant scalar curvature.
	In Section \ref{sec3}, we will apply $\Delta_f$ directly to the sum of the smallest $2$ eigenvalues, denoted $\lambda_1$ and $\lambda_2$ and then derive the estimate of $\Delta_f(\lambda_1+\lambda_2)$,  which involves the Weyl curvature of the level set as mentioned before, see Proposition \ref{lap12}.
	
	In Section \ref{sec4}, based on the point-picking argument, we prove the Riemannian curvature is bounded.
	In Section \ref{sec6}, we prove Theorem \ref{main}.
      In Section 6, we prove Theorem 1.2.
	
	\smallskip

	
	\section{Notations and basic formulas on gradient shrinking Ricci solitons}\label{sec2}
	
	In this section, we recall the notations and basic formulas on gradient shrinking Ricci solitons with constant scalar curvature. For details, we refer to \cite{Cao,Hamilton,Petersen-Wylie,Cheng-Zhou}.
	
	\smallskip
	Let  $(M, g)$ be an $n$-dimensional complete gradient shrinking Ricci soliton satisfying \eqref{soliton}.
	By scaling the metric $g$,  one can normalize $\lambda$ so that $\lambda=\frac{1}{2}$. In this paper, we always assume $\lambda=\frac{1}{2}$ and the gradient shrinking Ricci soliton equation is as follows,
	\begin{align}\label{soliton'}
	\text{Ric}+\nabla^2f=\frac{1}{2} g.
	\end{align}
	
	At first we recall some basic formulas which will be used throughout the paper:
	\begin{align}\label{second bianchi}
	d R=2 Ric(\nabla f,\cdot),
	\end{align}
	\begin{align}\label{R}
	R+\Delta f=\frac{n}{2},
	\end{align}\begin{align}\label{f he tidu}
	R+|\nabla f|^2=f,
	\end{align}\begin{align}\label{lr}
	\Delta_f R=R-2|Ric|^2,
	\end{align}\begin{align}\label{elliptic equation}
	\Delta_f R_{ij}=R_{ij}-2 R_{ikjl}R_{kl}.
	\end{align}
	where $\Delta _{f}=\Delta -\left\langle \nabla f,\nabla \right\rangle $ is the weighted Laplacian, and $\Delta _{f}$ acting on the function
	is self-adjoint on the space of square integrable functions with respect to the weighted measure $e^{-f}dv.$ In general, the weighted Laplacian $\Delta_f$ acting on tensors is given by $\Delta_f=\Delta-\nabla_{\nabla f}$.
	
	\smallskip

	Now we consider complete gradient shrinking Ricci solitons with constant scalar curvature $R$. In this case, the potential function $f$ is isoparametric and the isoparametric property plays a very important role.
	Concretely, the potential function $f$ can be renormalized, by replacing $f-R$  with $f$, so that $f:M\to [0, +\infty)$ satisfies
	\begin{equation}\label{iso1}
	|\nabla f|^2=f,
	\end{equation}
	which implies that $f$ is transnormal.
	Recall \eqref{R}
	\begin{align*}\label{R'}
	\Delta f=\frac{n}2-R,
	\end{align*}
	therefore the (nonconstant) renormalized $f$ is an isoparametric function on $M$. From the potential function estimate \eqref{cao-zhou},  $f$ is proper and unbounded.

For later application,  we need to discuss the relation between barrier solution and distribution.

Let $(M,g)$ be a complete manifold, $\nabla$ be the gradient operator on $M$,
$\Delta$ be the Laplacian on $M$, $\Omega$ be an open set of $M$ and
$F\in C^\infty(M)$. The $F$-Laplacian is defined by
\[
\Delta_F:=\Delta-\nabla F\cdot\nabla.
\]
We say that a continuous function $u\in C(\Omega)$ satisfies $\Delta_F u\le w$
for some $w$ \emph{in the barrier sense}, if for any fixed $x\in \Omega$,
there exists a smooth function $v$ defined in a neighborhood $U(x)$ of $x$,
such that $u(x)=v(x)$, $u(y)\le v(y)$ for any $y\in U(x)$ and
\[
\Delta_F v(x)\le w(x).
\]
We say that $u\in C(\Omega)$ satisfies $\Delta_F u\le w$ on $\Omega$
\emph{in the sense of distribution}, if
\[
\int_{\Omega} u\Delta_F\phi\, e^{-F}\leq \int_{\Omega} w \phi\, e^{-F}
\]
for any $\phi\geq 0$ with $\phi\in C_c^\infty(\Omega)$.

In the appendix of \cite{Wu-Wu}, the authors proved the following Theorem.
\begin{theo}[\cite{Wu-Wu}]\label{bard}
If $u\in C(\Omega)$ satisfies $\Delta_F u\le w$ for some $w$
in the barrier sense, then it satisfies $\Delta_F u\le w$ on $\Omega$ in
the sense of distribution.
\end{theo}
Next we shall apply Theorem \ref{bard} to discuss some special cases for our
theorems in the preceding sections. We use the same notations as before. On
an $n$-dimensional shrinker $(M,g,f)$, let $u$ be an
 Lipschitz function and let $w$ be an integrable function
with respect to measure $e^{-h}dv$. Here $h$ is uniformly equivalent to the square of
distance function. For a sufficiently large
$r$, $a$ and $b$ ($a<b$), set $D(r):=\{x\in M|h(x)\le r\}$,
$\Sigma(r):=\{x\in M|h(x)=r\}$ and $D(a,b):=\{x\in M|a\le h(x)\le b\}$.
Then Theorem \ref{bard} implies that
\begin{coro}[\cite{Wu-Wu}]\label{bardis}
If $\Delta_h u\le w$ holds on $M\setminus D(r)$ in the barrier sense, then it
holds on $M\setminus D(r)$ in the sense of distribution.
\end{coro}
Using Corollary \ref{bardis}, we have the following useful proposition.
\begin{prop}[\cite{Wu-Wu}]\label{keyprop}
If a continuous and Lipschitz function $u$ satisfies $\Delta_h u\le w$
on $M\setminus D(r)$ in the barrier sense, then
\begin{align*}
-\int_{\Sigma(r)}\langle\nabla u, \tfrac{\nabla h}{|\nabla h|}\rangle e^{-h}\le \int_{M\setminus D(r)} w\, e^{-h}
\end{align*}
holds for almost everywhere sufficiently large $r$.
\end{prop}
\Rk. During the paper,  $\int_{M\setminus D(a)} \Delta_h u\cdot e^{-h}dvol$ is understood as $-\int_{\Sigma(a)}\langle \nabla u, \frac{\nabla h}{|\nabla h|}\rangle e^{-h}$, for more details, see \cite{Wu-Wu}.

	%
	\section{Estimate on the sum of the smallest two Ricci-eigenvalues of weighted Laplacian operator}\label{sec3}
	In this section, let $(M, g, f)$ be an $n$-dimensional complete noncompact gradient shrinking Ricci soliton satisfying \eqref{soliton'} with constant scalar curvature $R=\frac{n-2}{2}$. We will apply $\Delta_f$ directly to the sum of the smallest $2$ Ricci-eigenvalues, denoted $\lambda_1$ and $\lambda_2$, and derive the estimate of $\Delta_f(\lambda_1+\lambda_2)$ involving the Weyl curvature of the level set, see Propositon \ref{lap12}.
	
	First, we recall the non-negativity of Ricci curvature for $n$-dimensional gradient shrinking Ricci soliton with constant scalar curvature $R=\frac{n-2}{2}$, also see \cite{FR16}.
	\begin{prop}[\cite{FR16}] \label{positive}
		Let $(M, g, f)$ be an $n$-dimensional complete noncompact gradient shrinking Ricci soliton satisfying \eqref{soliton'} with constant scalar curvature $R=\frac{n-2}{2}$. Then it has nonnegative Ricci curvature, the smallest Ricci-eigenvalue $\lambda_1=0$ and $Ric(\nabla f,\cdot)=0$.
	\end{prop}
	\begin{proof}
		Let $(M, g, f)$ be an $n$-dimensional complete noncompact gradient shrinking Ricci soliton satisfying \eqref{soliton'} with constant scalar curvature $R=\frac{n-2}{2}$, the authors proved that  the Ricci curvature is non-negative in \cite{FR16}.
		\smallskip
		\eqref{second bianchi} impiles $Ric(\nabla f,\cdot)=0$ since the scalar curvature is constant.
		Hence $\lambda_1=0$ is the smallest Ricci-eigenvalue with Ricci-eigenvector $\nabla f$.
		
	\end{proof}
	
	From Proposition \ref{positive}, throughout this paper we always denote the eigenvalues of Ricci curvature by
	\[
	0=\lambda_1\leq \lambda_2\leq  \lambda_3 \leq  \lambda_4\leq\cdot\cdot\cdot \leq \lambda_n.
	\]
	
	\smallskip
	In the following, we calculate $\Delta_f(\lambda_1+\lambda_2)$ in the barrier sense, and have the following lemma.
	\begin{lem} Let $(M^n, g, f)$ be an $n$-dimensional shrinking gradient Ricci soliton with constant scalar curvature $\frac{n-2}{2}$.  	
		Then
		\begin{equation}\label{l12}
		\Delta_f(\lambda_1+\lambda_2)\leq(\lambda_1+\lambda_2)-2\sum_{\alpha=2}^n K_{1\alpha}\lambda_\alpha-2\sum_{\alpha=3}^n K_{2\alpha}\lambda_\alpha
		\end{equation}	
		in the sense of barrier, where $K_{ij}$ denotes the sectional curvature of the plane spanned by $e_i$ and $e_j$, and $\{e_i\}_{i=1}^n$ are the orthonormal eigenvectors corresponding to the Ricci-eigenvalues $ \{\lambda_i\}_{i=1}^n$.
	\end{lem}
	
	\begin{proof}
		Actually, at $x$, because $R=\frac{n-2}{2}$, $Ric(\nabla f)=0$, so we choose $e_1=\frac{\nabla f}{|\nabla f|}$, then extend $e_1$ to an orthonormal basis $\{e_1, e_2, e_3,\cdot\cdot\cdot, e_n\}$ such that $\{e_i\}_{i=1}^n$ are the eigenvectors of $Ric(x)$ corresponding to eigenvalues $\{\lambda_i\}_{i=1}^n$. Take parallel transport of $\{e_i\}_{i=1}^n$ along all the geodesics from $x$, then in a neighborhood $B(x, \delta)$ we get a smooth function $u(y)=Ric(y)(e_1(y), e_1(y))+Ric(y)(e_2(y), e_2(y))$ satisfying $u(y)\geq \lambda_1(y)+\lambda_2(y)$ and  $u(x)= \lambda_1(x)+\lambda_2(x)$.
		Thus, at $x$,
	 \begin{equation*}
		\begin{aligned}
			\Delta_f u(x)
			=&\Delta_f\left(Ric(y)(e_1(y), e_1(y))+Ric(y)(e_2(y), e_2(y)) \right)|_{y=x}\\
			=&(\Delta_f Ric)(e_1, e_1)(x)+(\Delta_f Ric)(e_2, e_2)(x)\\
			=&(\lambda_1+\lambda_2)-2(\sum_{i=1}^n K_{1i}\lambda_i +\sum_{i=1}^n K_{2i}\lambda_i)\\
			=&(\lambda_1+\lambda_2)-2\sum_{\alpha=2}^n K_{1\alpha}\lambda_\alpha
			-2\sum_{\alpha=3}^n K_{2\alpha}\lambda_\alpha,
		\end{aligned}
	\end{equation*}
	this completes the proof.
\end{proof}
	
	Next, we deal with the term $\sum_{\alpha=2}^n K_{1\alpha}\lambda_\alpha$. First we will give the following lemma for preparation.
	
	\begin{lem} Let $(M^n, g, f)$ be an $n$-dimensional shrinking gradient Ricci soliton with constant scalar curvature $\frac{n-2}{2}$. Then
		\begin{equation}\label{k1a}
		K_{1\alpha}=\frac{\nabla_{\nabla f}R_{\alpha\alpha} +\lambda_\alpha(\frac{1}{2}-\lambda_\alpha)}{f}
		\end{equation}	
		for $\alpha=2,  3, \cdot\cdot\cdot, n$.	
	\end{lem}
	
	\begin{proof}
		From the Ricci identity, we have
		\begin{equation*}
		\begin{aligned}
		&-R(\nabla f, e_\alpha, \nabla f, e_\beta) \\
		=&-\left( \nabla_\beta f_{\alpha k} - \nabla_kf_{\alpha \beta}\right) f_{k}\\
		=&\left( \nabla_\beta R_{\alpha k} - \nabla_kR_{\alpha \beta}\right) f_{k}\\
		=&-  \nabla_{\nabla f} R_{\alpha \beta}+\nabla_{\beta}(R_{\alpha k} f_k)- R_{\alpha k}f_{ k\beta}\\
		=&-  \nabla_{\nabla f} R_{\alpha \beta}- R_{\alpha k}\left( \frac{1}{2}g_{ k\beta}-R_{ k\beta} \right) \\
		=&- \nabla_{\nabla f} R_{\alpha \beta}-\left( \frac{1}{2}R_{\alpha \beta}-\sum_{k=1}^nR_{\alpha k}R_{k\beta}\right),
		\end{aligned}
		\end{equation*}
		where \eqref{second bianchi} was used in the third equality. Therefore, we see
		\begin{equation}\label{R1a1b}
		R(e_1, e_\alpha, e_1, e_\beta)=\frac{ \nabla_{\nabla f} R_{\alpha \beta}+\left( \frac{1}{2}R_{\alpha \beta}-\sum_{k=1}^n R_{\alpha k}R_{k\beta}\right) }{f}
		\end{equation}
		due to $|\nabla f|^2=f$. \eqref{k1a} holds by setting $\beta=\alpha$ in \eqref{R1a1b}.
		This completes the proof of the lemma.
	\end{proof}
	
	\begin{lem}\label{lek1a}
		Let $(M^n, g, f)$ be an  $n$-dimensional shrinking gradient Ricci soliton with constant scalar curvature $\frac{n-2}{2}$. Then we have
		\ban
		 -\sum_{\alpha=2}^nK_{1\alpha}\lambda_\alpha=-\frac{1}{f}\sum_{\alpha=2}^n\lambda_\alpha^2(\frac{1}{2}-\lambda_\alpha)=\frac{1}{f}\sum_{\alpha=2}^n(\lambda_\alpha-\frac{1}{2})^2\lambda_\alpha.
		\ean
	\end{lem}
	
	\begin{proof}
		First, since the scalar curvature is constant, \eqref{lr} implies
		\ban
		|Ric|^2=\frac{R}{2}=\frac{n-2}{4}.
		\ean
		This means $\sum_{\alpha=2}^n\lambda_\alpha=\lambda_2+\lambda_3+\cdots+\lambda_n=\frac{n-2}{2}$ and $\sum_{\alpha=2}^n\lambda_\alpha^2=\lambda_2^2+\lambda_3^2+\cdots+\lambda_n^2=\frac{n-2}{4}$. Hence
		\ban
		\sum_{\alpha=2}^n\lambda_\alpha(\frac{1}{2}-\lambda_\alpha)=0.
		\ean
		Recall (\ref{k1a}), and we obtain
		\ban
		&&-\sum_{\alpha=2}^nK_{1\alpha}\lambda_\alpha\\
		&=&-\frac{1}{f}\sum_{\alpha=2}^n
		\left[ \nabla f\cdot \nabla \lambda_\alpha +\lambda_\alpha(\frac{1}{2}-\lambda_\alpha)\right]\lambda_\alpha\\
		&=&-\frac{1}{f}\left[ \frac{1}{2} \nabla f\cdot \sum_{\alpha=2}^n\lambda_\alpha^2
		+\sum_{\alpha=2}^n\lambda_\alpha^2(\frac{1}{2}-\lambda_\alpha) \right] \\
		&=&-\frac{1}{f}\sum_{\alpha=2}^n\lambda_\alpha^2(\frac{1}{2}-\lambda_\alpha) \\
		&=&-\frac{1}{f}\sum_{\alpha=2}^n\left[- \lambda_\alpha(\frac{1}{2}-\lambda_\alpha)^2+\frac{1}{2}\lambda_\alpha(\frac{1}{2}-\lambda_\alpha) \right] \\
		&=&\frac{1}{f}\sum_{\alpha=2}^n \lambda_\alpha(\frac{1}{2}-\lambda_\alpha)^2.
		\ean
		We have completed the proof of this Lemma.
		
	\end{proof}

	Subsequently, we consider the level set $\Sigma$ of the potential function $f$ to handle the term $-2\sum_{\alpha=3}^n K_{2\alpha}\lambda_\alpha$. For this purpose,
	recall that the intrinsic curvature tensor $R^{\Sigma}_{\alpha \beta \gamma \eta }$ and the extrinsic curvature tensor $R_{\alpha \beta \gamma \eta }$ of ${\Sigma}$ where $\{\alpha,\beta,\gamma,\eta\}\in\{2,3,\cdot\cdot\cdot, n\}$, are related by the Gauss equation:
	\ban
	R^{\Sigma}_{\alpha \beta \gamma \eta }=R_{\alpha \beta \gamma \eta }+h_{\alpha \gamma}h_{\beta \eta}-h_{\alpha \eta }h_{\beta \gamma},
	\ean
	where $h_{\alpha \beta }$ denotes the components of the second fundamental form $A$ of ${\Sigma}$. Moreover,
	
	\smallskip
	\textbf{Claim}
	\ba\label{rsigma}
	R^{\Sigma}=R=\frac{n-2}{2};
	\ea
	\ba\label{ricsigma}
	Ric^{\Sigma}=Ric+\frac{\nabla_{\nabla f} Ric}{f};
	\ea
	\ba\label{ricsigma'}
	|Ric^{\Sigma}|^2=|Ric|^2+\frac{|\nabla_{\nabla f} Ric|^2}{f^2}.
	\ea	
	
	\begin{equation}\label{kab}
	\begin{aligned}
	K_{\alpha \beta }=&{\frac{1}{n-3}}(\lambda_{\alpha}+\lambda_{\beta})-{\frac{1}{2(n-3)}}-{\frac{1}{n-3}}(K_{1\alpha }+K_{1\beta })+{\frac{1}{2(n-3)f}}[1-(\lambda_{\alpha }+\lambda_{\beta})] \\
	&- {\frac{1}{(n-3)f}}\left[ ({\frac{1}{2}}-\lambda_{\alpha})^2+({\frac{1}{2}}-\lambda_{\beta})^2\right]-{\frac{1}{f}}({\frac{1}{2}}-\lambda_{\alpha})({\frac{1}{2}}-\lambda_{\beta})+W^{\Sigma}_{\alpha \beta},
	\end{aligned}
	\end{equation}
	where $W^{\Sigma}_{\alpha \beta}=W^{\Sigma}_{\alpha \beta\alpha \beta} $ denotes the components of the Weyl curvature of ${\Sigma}$.	\\

	In fact, it follows from the Gauss equation that
	\ban
	R^{\Sigma}_{\alpha  \beta }=R_{\alpha  \beta }-R_{1\alpha 1 \beta }+Hh_{\alpha  \beta }-h_{\alpha\gamma }h_{\gamma\beta}
	\ean
	and the scalar curvature $R^{\Sigma}$ of ${\Sigma}$ satisfies
	\ban
	R^{\Sigma}=R-2R_{11}+H^2-|A|^2.
	\ean
	Since $R=\frac{n-2}{2}$, $Ric(\nabla f, \cdot)=0$, $R_{1i}=0$, $i=1,\cdots ,n$, then
	\ban
	R^{\Sigma}=R+H^2-|A|^2.
	\ean
	Noting
	\ban
	h_{\alpha \beta }=\frac{f_{\alpha \beta }}{|\nabla f|}=\frac{\frac{1}{2}-\lambda_\alpha}{\sqrt{f}}\delta_{\alpha \beta},
	\ean
	then the mean curvature satisfies
	\[
	H=\frac{\frac{n-1}{2}-\sum\lambda_\alpha}{\sqrt{f}}=\frac{1}{2\sqrt{f}}
	\]
	and
	\begin{align*}
	|A|^2=&\frac{1}{f}\sum(\frac{1}{2}-\lambda_\alpha)^2=\frac{1}{f}(1-\sum\lambda_\alpha+\sum \lambda_\alpha^2)\\
	=&\frac{1}{f}(\frac{n-1}{4}-\frac{n-2}{2}+\frac{n-2}{4})=\frac{1}{4f}.
	\end{align*}
	Hence, $R^{\Sigma}=R=\frac{n-2}{2}$. Together with \eqref{R1a1b}, we see
	\begin{equation*}
	\begin{aligned}
	R^{\Sigma}_{\alpha  \beta }=&R_{\alpha  \beta }-\frac{\nabla_{\nabla f} R_{\alpha \beta}+\left( \frac{1}{2}R_{\alpha \beta}-R_{\alpha k}R_{k\beta}\right) }{f}
	+\frac{\frac{1}{2}-\lambda_\alpha}{2f}\delta_{\alpha \beta}\\
	&-\frac{(\frac{1}{2}-\lambda_\alpha)(\frac{1}{2}-\lambda_\beta)}{f}\delta_{\alpha \gamma}\delta_{\gamma\beta}\\
	=&R_{\alpha  \beta }-\frac{ \nabla_{\nabla f} R_{\alpha \beta}}{f}+\frac{1}{f}[-\lambda_\alpha(\frac{1}{2}-\lambda_\alpha)+\frac{1}{2}(\frac{1}{2}-\lambda_\alpha)-(\frac{1}{2}-\lambda_\alpha)^2]\delta_{\alpha \beta}\\
	=&R_{\alpha  \beta }-\frac{\nabla _{\nabla f}  R_{\alpha \beta}}{f},
	\end{aligned}
	\end{equation*}
	which implies \eqref{ricsigma} holds.
	
	\smallskip
	By $|Ric|^2=\frac{n-2}{4}$ agian, we have $Ric\cdot\nabla_{\nabla f}Ric=0$.
	\ban
	|Ric^{\Sigma}|^2&&=|Ric|^2+\frac{|\nabla_{\nabla f} Ric|^2}{f^2}-2\frac{Ric\cdot\nabla_{\nabla f}Ric}{f}\\
	&&=|Ric|^2+\frac{|\nabla_{\nabla f} Ric|^2}{f^2}.
	\ean
	
	Finally, recall that the relationship between curvature and the Weyl curvature
	\ban
	R_{ijkl}=&&W_{ijkl}+{\frac{1}{n-3}}(g_{ik}R_{jl}-g_{il}R_{jk}-g_{jk}R_{il}+g_{jl}R_{ik})\\
	&&-{\frac{1}{(n-2)(n-3)}}R(g_{ik}g_{jl}-g_{il}g_{jk}),
	\ean
	and we get
	\begin{equation*}
	\begin{aligned}
	K^{\Sigma}_{\alpha \beta }=&{\frac{1}{n-3}}(R^{\Sigma}_{\alpha \alpha }+R^{\Sigma}_{\beta \beta})-{\frac{1}{(n-2)(n-3)}}R^{\Sigma}\\
	=&{\frac{1}{n-3}}(R_{\alpha \alpha}-K_{1\alpha }+Hh_{\alpha \alpha}-h_{\alpha \alpha}^2+R_{\beta \beta}-K_{1\beta}+Hh_{\beta \beta }-h_{\beta \beta}^2)-{\frac{1}{2(n-3)}}\\
	=&{\frac{1}{n-3}}\left(\lambda_{\alpha}+\lambda_{\beta}-K_{1\alpha }-K_{1\beta  }+H(h_{\alpha \alpha}+h_{\beta \beta })-h_{\alpha \alpha}^2-h_{\beta \beta}^2\right)-{\frac{1}{2(n-3)}}+W^{\Sigma}_{\alpha \beta }.
	\end{aligned}
	\end{equation*}
	on the  hypersurface $\Sigma$. It follows from the Gauss equation that
	\begin{equation*}
	\begin{aligned}
	K_{\alpha \beta }=&K^{\Sigma}_{\alpha \beta }-h_{\alpha \alpha}h_{\beta \beta}+h_{\alpha \beta }^2\\
	=&{\frac{1}{n-3}}\left(\lambda_{\alpha}+\lambda_{\beta}-K_{1\alpha }-K_{1\beta }+H(h_{\alpha \alpha}+h_{\beta \beta })-h_{\alpha \alpha}^2-h_{\beta \beta}^2\right)\\
	&-{\frac{1}{2(n-3)}}-h_{\alpha \alpha}h_{\beta \beta}+W^{\Sigma}_{\alpha \beta }\\
	=&{\frac{1}{n-3}}(\lambda_{\alpha}+\lambda_{\beta})-{\frac{1}{2(n-3)}}-{\frac{1}{n-3}}(K_{1\alpha }+K_{1\beta  })+{\frac{1}{2(n-3)f}}[({\frac{1}{2}}-\lambda_{\alpha })+({\frac{1}{2}}-\lambda_{\beta})]\\
	&- {\frac{1}{(n-3)f}}\left[ ({\frac{1}{2}}-\lambda_{\alpha})^2+({\frac{1}{2}}-\lambda_{\beta})^2\right]  -{\frac{1}{f}}({\frac{1}{2}}-\lambda_{\alpha})({\frac{1}{2}}-\lambda_{\beta})+W^{\Sigma}_{\alpha \beta }.
	\end{aligned}
	\end{equation*}
	We have completed the proof of equations \eqref{rsigma}-\eqref{kab} in \textbf{Claim}.\\

	Using equations \eqref{rsigma}-\eqref{kab}, we have the following result.
	
	\begin{lem}\label{lek2a}
		Let $(M^n, g, f)$ be an $n$-dimensional shrinking gradient Ricci soliton with constant scalar curvature $\frac{n-2}{2}$. Suppose each level set of $f$ has vanishing Weyl curvature, then
 \begin{equation*}
	\begin{aligned}
		&-2\sum_{\alpha=3}^n K_{2\alpha}\lambda_\alpha\\
		=&-\frac{1}{n-3}(\lambda_1+\lambda_2)[(n-1)-4(\lambda_1+\lambda_2)] +\frac{n-2}{(n-3)f}\nabla f\cdot \nabla (\lambda_1+\lambda_2)\\
		&-\frac{2}{(n-3)f}\nabla f\cdot \nabla(\lambda_1+\lambda_2)^2
		-\frac{2}{f}(\lambda_1+\lambda_2)({\lambda_1+\lambda_2-\frac{1}{2}})^2  _.
	\end{aligned}
\end{equation*}
\end{lem}

\begin{proof}
It follows from \eqref{kab} that
\begin{equation*}
	\begin{aligned}
		&-2\sum_{\alpha=3}^nK_{2\alpha}\lambda_\alpha\\
		 =&\frac{-2}{n-3}\sum_{\alpha=3}^n(\lambda_{2}+\lambda_{\alpha})\lambda_\alpha+\frac{1}{n-3}\sum_{\alpha=3}^n\lambda_\alpha+\frac{2}{n-3}\sum_{\alpha=3}^n(K_{12}+K_{1\alpha })\lambda_\alpha\\
		&-{\frac{1}{(n-3)f}}\sum_{\alpha=3}^n\left[ ({\frac{1}{2}}-\lambda_{2})+({\frac{1}{2}}-\lambda_{\alpha})\right] \lambda_\alpha\\
		&+\frac{2}{(n-3)f}\left[({\frac{1}{2}}-\lambda_{2})^2\sum_{\alpha=3}^n\lambda_{\alpha}+ \sum_{\alpha=3}^n({\frac{1}{2}}-\lambda_{\alpha})^2 \lambda_{\alpha}\right]\\
		&+\frac{2}{f}({\frac{1}{2}}-\lambda_{2})\sum_{\alpha=3}^n({\frac{1}{2}}-\lambda_{\alpha})\lambda_{\alpha}
		\\
	\end{aligned}
\end{equation*}

Next, we handle these items one by one. First, notice that $\sum_{\alpha=3}^n\lambda_\alpha=\frac{n-2}{2}-\lambda_2$ and $\sum_{\alpha=3}^n\lambda_\alpha^2=\frac{n-2}{4}-\lambda_2^2$, it is easy to check
\begin{equation*}
	\begin{aligned}
		&-\frac{2}{n-3}\sum_{\alpha=3}^n(\lambda_{2}+\lambda_{\alpha})\lambda_\alpha+\frac{1}{n-3}\sum_{\alpha=3}^n\lambda_\alpha\\		 =&-\frac{2}{n-3}(\lambda_{2}\sum_{\alpha=3}^n\lambda_\alpha+\sum_{\alpha=3}^n\lambda_\alpha^2)+\frac{1}{n-3}(\frac{n-2}{2}-\lambda_{2})\\
		 =&-\frac{2}{n-3}\left[\lambda_2\left(\frac{n-2}{2}-\lambda_2\right)+\left(\frac{n-2}{4}-\lambda_2^2\right)\right]+\frac{1}{n-3}\left(\frac{n-2}{2}-\lambda_2\right)\\
		=&\frac{4}{n-3}\lambda_2^2+\frac{1-n}{n-3}\lambda_2.\\
	\end{aligned}
\end{equation*}

It follows from \eqref{k1a} that
\begin{equation*}
	\begin{aligned}
		&\frac{2}{n-3}\sum_{\alpha=3}^n(K_{12}+K_{1\alpha })\lambda_\alpha\\
		=& \frac{2}{n-3}K_{12}({\frac{n-2}{2}}-2\lambda_{2})+\frac{2}{n-3}\sum_{\alpha=2}^nK_{1\alpha}\lambda_\alpha\\
		=&\frac{2}{n-3}({\frac{n-2}{2}}-2\lambda_{2})\frac{1}{f}\left[\nabla f\cdot \nabla \lambda_2 +\lambda_2(\frac{1}{2}-\lambda_2) \right]+\frac{2}{n-3}\sum_{\alpha=2}^nK_{1\alpha}\lambda_\alpha\\
		=&\frac{n-2}{(n-3)f}\nabla f\cdot \nabla \lambda_2-\frac{2}{(n-3)f}\nabla f\cdot \nabla \lambda_2^2
		+\frac{2}{(n-3)f}\lambda_2(\frac{1}{2}-\lambda_2)({\frac{n-2}{2}}-2\lambda_{2})\\
		&+\frac{2}{(n-3)}\sum_{\alpha=2}^nK_{1\alpha}\lambda_\alpha.
	\end{aligned}
\end{equation*}

Using the fact that $\sum_{\alpha=3}^n\lambda_\alpha=\frac{n-2}{2}-\lambda_2$ and $\sum_{\alpha=3}^n\lambda_\alpha^2=\frac{n-2}{4}-\lambda_2^2$,
we have
\begin{align*}
\sum_{\alpha=3}^n \lambda_\alpha(\frac{1}{2}-\lambda_\alpha)=\frac{1}{2}(\frac{n-2}{2}-\lambda_2)-(\frac{n-2}{4}-\lambda_2^2)=\lambda_2(\lambda_2-\frac{1}{2}).
\end{align*}

Hence
\begin{equation*}
	\begin{aligned}
		&-{\frac{1}{(n-3)f}}\sum_{\alpha=3}^n\left[ ({\frac{1}{2}}-\lambda_{2})+({\frac{1}{2}}-\lambda_{\alpha})\right] \lambda_\alpha\\	 &+\frac{2}{(n-3)f}(\frac{1}{2}-\lambda_{2})^2\sum_{\alpha=3}^n\lambda_{\alpha}+\frac{2}{f}({\frac{1}{2}}-\lambda_{2})\sum_{\alpha=3}^n({\frac{1}{2}}-\lambda_{\alpha})\lambda_{\alpha}\\
		 =&-\frac{1}{(n-3)f}\left[(\frac{1}{2}-\lambda_2)(\frac{n-2}{2}-\lambda_2)+\lambda_2(\lambda_2-\frac{1}{2})\right] \\
		  &+\frac{2}{(n-3)f}\left(\frac{1}{2}-\lambda_2\right)^2(\frac{n-2}{2}-\lambda_2)+\frac{2}{f}(\frac{1}{2}-\lambda_2)\lambda_2(\lambda_2-\frac{1}{2})\\
		 =&\frac{1}{(n-3)f}(\frac{1}{2}-\lambda_2)(\frac{n-2}{2}-\lambda_2)(-1+1-2\lambda_2)+\frac{1}{f}(\frac{-1}{n-3}+1-2\lambda_2)\lambda_2(\lambda_2-\frac{1}{2})\\
		 =&\frac{-2\lambda_2}{(n-3)f}(\frac{1}{2}-\lambda_2)(\frac{n-2}{2}-\lambda_2)+\frac{1}{f}(\frac{n-4}{n-3}-2\lambda_2)\lambda_2(\lambda_2-\frac{1}{2})\\
=&\frac{1}{f}\lambda_2(\lambda_2-\frac{1}{2})(2-\frac{2n-4}{n-3}\lambda_2).
	\end{aligned}
\end{equation*}

Therefore, we have
\begin{equation*}
	\begin{aligned}
		&-2\sum_{\alpha=3}^nK_{2\alpha}\lambda_\alpha\\
	=&\frac{4}{n-3}\lambda_2^2+\frac{1-n}{n-3}\lambda_2+\frac{n-2}{(n-3)f}\nabla f\cdot \nabla \lambda_2-\frac{2}{(n-3)f}\nabla f\cdot \nabla \lambda_2^2+\frac{2}{n-3}\sum_{\alpha=2}^nK_{1\alpha}\lambda_\alpha\\
		&+\frac{2}{(n-3)f}\sum_{\alpha=3}^n({\frac{1}{2}}-\lambda_{\alpha})^2\lambda_{\alpha}	 +\frac{2}{(n-3)f}\lambda_2(\frac{1}{2}-\lambda_2)({\frac{n-2}{2}}-2\lambda_{2})\\
		&+\frac{1}{f}\lambda_2(\lambda_2-\frac{1}{2})(2-\frac{2n-4}{n-3}\lambda_2)\\
=&\frac{4}{n-3}\lambda_2^2+\frac{1-n}{n-3}\lambda_2+\frac{n-2}{(n-3)f}\nabla f\cdot \nabla \lambda_2-\frac{2}{(n-3)f}\nabla f\cdot \nabla \lambda_2^2-\frac{2}{(n-3)f}({\frac{1}{2}}-\lambda_2)^2\lambda_2\\
 &+	\frac{2}{(n-3)f}\lambda_2(\frac{1}{2}-\lambda_2)({\frac{n-2}{2}}-2\lambda_{2})
		+\frac{1}{f}\lambda_2(\lambda_2-\frac{1}{2})(2-\frac{2n-4}{n-3}\lambda_2)\\
=&\frac{4}{n-3}\lambda_2^2+\frac{1-n}{n-3}\lambda_2+\frac{n-2}{(n-3)f}\nabla f\cdot \nabla \lambda_2-\frac{2}{(n-3)f}\nabla f\cdot \nabla \lambda_2^2\\
&+\frac{1}{f}\lambda_2(\lambda_2-\frac{1}{2})\left( \frac{-2}{n-3}(\lambda_2-\frac{1}{2})-\frac{2}{n-3}(\frac{n-2}{2}-2\lambda_2)+2-\frac{2n-4}{n-3}\lambda_2\right)\\
=&\frac{4}{n-3}\lambda_2^2+\frac{1-n}{n-3}\lambda_2+\frac{n-2}{(n-3)f}\nabla f\cdot \nabla \lambda_2-\frac{2}{(n-3)f}\nabla f\cdot \nabla \lambda_2^2-\frac{2}{f}\lambda_2(\lambda_2-\frac{1}{2})^2.
	\end{aligned}
\end{equation*}
The proof of Lemma \ref{lek2a} was thus completed.
\end{proof}
	
	We have the following two basic inequalities to obtain our desired estimate.
	\begin{lem}
		Let $(M^n, g, f)$ be an $n$-dimensional shrinking gradient Ricci soliton with constant scalar curvature $\frac{n-2}{2}$. Then
		\begin{equation} \label{la}
		(\frac{1}{2}-\lambda_\alpha)^2\leq\sum_{\alpha=3}^n(\frac{1}{2}-\lambda_\alpha)^2=\lambda_2(1-\lambda_2)
		\end{equation}
		for $\alpha=3,4,\cdots, n$;
		\begin{equation}\label{l345}
		\sum_{\alpha=3}^n(\frac{1}{2}-\lambda_\alpha)^2\lambda_\alpha\leq\lambda_2(1-\lambda_2)(\frac{n-2}{2}-\lambda_2);
		\end{equation}
		\begin{equation}\label{l345'}
		\sum_{\alpha=3}^n(\frac{1}{2}-\lambda_\alpha)^2\lambda_\alpha^2\leq\lambda_2(1-\lambda_2)(\frac{n-2}{4}-\lambda_2^2).
		\end{equation}
	\end{lem}

	\begin{proof}
		It follows from \eqref{lr} that
		\begin{equation*}
		\begin{aligned}
		0=&\sum_{\alpha=2}^n\lambda_{\alpha}({\frac{1}{2}}-\lambda_{\alpha})\\
		=&-\sum_{\alpha=2}^n\left[ ({\frac{1}{2}}-\lambda_{\alpha})^2+\frac{1}{2}({\frac{1}{2}}-\lambda_{\alpha})\right] \\
		=&-\sum_{\alpha=2}^n({\frac{1}{2}}-\lambda_\alpha)^2+\frac{1}{4},
		\end{aligned}
		\end{equation*}
		which yeilds that
		\begin{equation*}
		\begin{aligned}
		\sum_{\alpha=3}^n({\frac{1}{2}}-\lambda_{\alpha})^2
		=&\frac{1}{4}-({\frac{1}{2}}-\lambda_2)^2\\
		=&\lambda_2(1-\lambda_2).		
		\end{aligned}
		\end{equation*}
		Therefore, we have
		\begin{equation*}
		\begin{aligned}
		&\,\sum_{\alpha=3}^n(\frac{1}{2}-\lambda_\alpha)^2\lambda_\alpha\\
		\leq&\, \lambda_2(1-\lambda_2) (\lambda_3+\lambda_4+\cdot\cdot\cdot+\lambda_n)\\
		=&\,\lambda_2(1-\lambda_2)(\frac{n-2}{2}-\lambda_2).		
		\end{aligned}
		\end{equation*}
		Similarly, \eqref{l345'} holds.
		
	\end{proof}
	
	Finally, we have the following proposition.
	
	\begin{prop}\label{lap12}
		Let $(M^n, g, f)$ be an $n$-dimensional shrinking gradient Ricci soliton with constant scalar curvature $R=\frac{n-2}{2}$.  Suppose each level set of $f$ has vanishing Weyl curvature, then we have the following inequality holds in the barrier sense,
		\begin{equation*}
		\begin{aligned}
		&\Delta_f(\lambda_1+\lambda_2)\\
		&\leq\frac{-2}{n-3}(\lambda_1+\lambda_2)+\frac{4}{n-3}(\lambda_1+\lambda_2)^2+\frac{n-2}{(n-3)f}\nabla f\cdot \nabla (\lambda_1+\lambda_2)\\
&-\frac{2}{(n-3)f}\nabla f\cdot \nabla (\lambda_1+\lambda_2)^2+\frac{2}{f}(\lambda_1+\lambda_2)\left[1-(\lambda_1+\lambda_2)\right]\left[\frac{n-2}{2}-(\lambda_1+\lambda_2)\right]
		\end{aligned}
		\end{equation*}	
		on $M\setminus D(a)$ for some $a>0$.	
	\end{prop}
	
	\begin{proof} By inserting equations from Lemma \ref{lek2a} and \ref{lek1a} into \eqref{l12}, we obtain:
		\begin{equation*}
		\begin{aligned}
		&\Delta_f(\lambda_1+\lambda_2)\\
		\leq&\lambda_2-2\sum_{\alpha=2}^nK_{1\alpha}\lambda_\alpha-2\sum_{\alpha=3}^nK_{2\alpha}\lambda_\alpha\\
		=&\lambda_2+\frac{4}{n-3}\lambda_2^2+\frac{1-n}{n-3}\lambda_2+\frac{n-2}{(n-3)f}\nabla f\cdot \nabla \lambda_2-\frac{2}{(n-3)f}\nabla f\cdot \nabla \lambda_2^2\\
&-\frac{2}{f}\lambda_2(\lambda_2-\frac{1}{2})^2+\frac{2}{f}\sum_{\alpha=2}^n \lambda_\alpha(\lambda_\alpha-\frac{1}{2})^2\\
=&\frac{-2}{n-3}\lambda_2+\frac{4}{n-3}\lambda_2^2+\frac{n-2}{(n-3)f}\nabla f\cdot \nabla \lambda_2-\frac{2}{(n-3)f}\nabla f\cdot \nabla \lambda_2^2+\frac{2}{f}\sum_{\alpha=3}^n\lambda_\alpha(\lambda_\alpha-\frac{1}{2})^2\\
\leq &\frac{-2}{n-3}\lambda_2+\frac{4}{n-3}\lambda_2^2+\frac{n-2}{(n-3)f}\nabla f\cdot \nabla \lambda_2-\frac{2}{(n-3)f}\nabla f\cdot \nabla \lambda_2^2+\frac{2}{f}\lambda_2(1-\lambda_2)(\frac{n-2}{2}-\lambda_2),
		\end{aligned}
		\end{equation*}	
		where \eqref{l345} was used in the last inequality, and the proof of this proposition was completed.
	\end{proof}

	\section{Boundedness of Riemannian curvature}\label{sec4}
In this section, we prove the curvature is bounded, and the proof is similar to the argument in \cite{Li-Ou-Qu-Wu}. For completeness, we present the detailed proof.

\begin{lem}[\cite{Chen-Zhu}]\label{point picking}
	Given a complete noncompact Riemannian manifold with unbounded curvature, we can find a sequence of points $p_j$ divergent to infinity such that for each positive integer $j$, we have $|Rm(p_j)|\geq j$ and
	\ban
	|Rm(x)|\leq 4 |Rm(p_j)|
	\ean
	for $x\in B(p_j, \frac{j}{\sqrt{|Rm(p_j)|}})$.
\end{lem}
The following backward pseudolocality Theorem will also be useful.
\begin{theo}[\cite{Li-Wang2}]\label{backward pseudolocality} For any $\alpha>0$, there is $\epsilon(n, \alpha)$ such that the following holds.

 Let $(M^n, g(t))_{t\in I}$ be a Ricci flow induced by a shrinking gradient Ricci soliton. Given $(x_0, t_0)\in M\times I$ and $r>0$, if
 \ban
 |B_{t_0}(x_0, r)|\geq \alpha r^n, \quad |Rm|\leq (\alpha r)^{-2} \quad on \quad B_{t_0}(x_0, r),
 \ean
 then
 \ban
 |Rm|\leq (\epsilon r)^{-2}  \quad on \quad P(x_0, t_0; (1-\alpha)r, -(\epsilon r)^2, 0).
 \ean
\end{theo}

\begin{theo}\label{bounded curvature}Let $(M^n, g, f)$ be an $n$-dimensional shrinking gradient Ricci soliton with constant scalar curvature $R=\frac{n-2}{2}$.  Suppose each level set of $f$ has vanishing weyl curvature, then its curvature is bounded.
\end{theo}

\begin{proof}
	Suppose not, by Lemma \ref{point picking}, then there exists a sequence of points $p_j$ divergent to infinity such that for each positive integer $j$, we have $|Rm(p_j)|\geq j$ and
	\ban
	|Rm(x)|\leq 4 |Rm(p_j)|
	\ean
	for $x\in B(p_j, \frac{j}{\sqrt{|Rm(p_j)|}})$.
By the $\kappa$-noncollapsed theorem in \cite{Li-Wang} and the scalar curvature is bounded for $(M, g)$, we get that $\text{Vol}(B(p_i, 1))$ has a uniform positive lower bound \cite{Li-Wang}.	

For the shrinking gradient Ricci soliton  $(M,g,f)$,
 the associated Ricci flow $(M, g(t), p_j)$ is defined on $(-\infty, 0)$, where
\[
g(t):=(-t)\phi^*_tg,\quad \frac{d\phi_t}{dt}=\frac{\nabla f}{-t}
\quad  \text{and}\quad \phi_{-1}=\mathrm{Id}.
\]
Now we can use Theorem \ref{backward pseudolocality} with $r=|Rm|(p_j)^{-\frac{1}{2}}$ to derive that
\ban
|Rm|\leq (\epsilon r)^{-2} \quad on \quad B(p_j, g(-1), \frac{j}{2}r) \times [-1-(\epsilon r)^2, -1].
\ean
	  Then we can apply Hamilton's compactness theorem to obtain that the rescaled manifolds $\left(B(p_j, g, \frac{j}{\sqrt{|Rm(p_j)|}}),  |Rm(p_j)|g, p_j\right)$ converge to a smooth complete Riemannian manifold $(M_\infty, g_\infty, p_\infty)$ with $|Rm(p_\infty)|=1$ which is Ricci flat because $(M^n, g)$ has bounded Ricci curvature and $|Rm(p_j)|\rightarrow \infty$, moreover  $(M_\infty, g_\infty, p_\infty)$ has Euclidean volume growth.
	
	\smallskip
	Since the integral curves of $f$ passing through $p_j$ are geodesics with respect to $(M, g)$, the geodesic segment of these curves contained in $B\left( p_j, \frac{j}{\sqrt{|Rm(p_j)|}}\right) $ will converge to a geodesic line in $(M_\infty, g_\infty)$, then Cheeger-Gromoll's splitting theorem \cite{Cheeger-Gromoll} implies that $M_\infty =\mathbb{R}\times N^{n-1}$, where $N^{n-1}$ is Ricci flat and of Euclidean Volume growth.
	
	\smallskip
	Because $(\Sigma(s), g)$ has second fundamental form
	\ban
	h=\frac{\frac{1}{2}g-Ric}{|\nabla f|},
	\ean
	which tends to zero as $f\rightarrow \infty$,
	so the second fundamental form of $B\left( p_j, g, \frac{j}{\sqrt{|Rm(p_j)|}}\right)   \cap \Sigma(f(p_j))$ with metric $|Rm(p_j)|g$ converges to zero. This implies the level set $B\left( p_j, g, \frac{j}{\sqrt{|Rm(p_j)|}}\right)   \cap \Sigma(f(p_j))$ with the induced rescaled metrics $|Rm(p_j)|g$ will converge to $N^{n-1}$.
	
	Since each level set has vanishing Weyl curvature,
	 $N^{n-1}$ has vanishing Weyl curvature, hence it is flat. This contradicts the fact that $|Rm(p_\infty)|=1$.
\end{proof}

\begin{theo}
	Let $(M^n, g, f)$ be an $n$-dimensional shrinking gradient Ricci soliton with constant scalar curvature $R=\frac{n-2}{2}$.  Suppose each level set of $f$ has vanishing Weyl curvature, then $\lambda_1+\lambda_2\rightarrow 0$ at infinity.
\end{theo}

\begin{proof}
	Suppose on the contrary, then there exists a sequence of $q_j$ divergent to infinity with $(\lambda_1+\lambda_2)(q_j)\geq \delta$ for some $\delta>0$.
	
	\smallskip
	As \cite{Naber}, define $f_j(x)=\frac{f(x)-f(q_j)}{|\nabla f(q_j)|}$. By Theorem \ref{bounded curvature}, we have the curvature is bounded, so the pointed manifolds $(M, g, q_j)$ converge in Cheeger-Gromov sense to $(M_\infty, g_\infty, q_\infty)$. Since  $|\nabla f_j(q_j)|=1$,
	\ban
	\nabla^2 f_j=\frac{\nabla^2 f}{|\nabla f(q_j)|}=\frac{\frac{1}{2}g-Ric}{|\nabla f(q_j)|}
	\ean
	tends to zero at infinity,  $f_j$ converges to a smooth function $f_\infty$ with $|\nabla f_\infty|(q_\infty)=1$ and $\nabla^2 f_\infty =0$. Hence $M_\infty=\mathbb{R}\times N^{n-1}$, where $N^{n-1}$ is an $n-1$-dimensional complete Riemannian manifold with $Ric\geq 0$ and of constant scalar curvature $\frac{n-2}{2}$. Since each level set of $f$ has vanishing Weyl curvature, then $N^{n-1}$ also has vanishing weyl curvature.

	\smallskip
	Thanks to the classification of complete locally conformally flat manifolds
	with nonnegative Ricci curvature by Zhu \cite{Zhu} and Carron-Herzlich \cite{Carron-Herzlic}, $N^{n-1}$ is
	one of the following:
	
	\smallskip
	(1) $N$ is non-flat and globally conformally equivalent to $\mathbb{R}^{n-1}$;
	
	\smallskip
	(2)  $N$ is globally conformally equivalent to a space form of positive curvature;
	
	\smallskip
	(3)  $N$ is locally isometric to the cylinder $\mathbb{R}\times \mathbb{S}^{n-2}$;
	
	\smallskip
	(4)   $N$ is isometric to a complete flat manifold.\\
	
	If  case (1) happens, then there is a positive function $u$ such that $g_N=u^{\frac{4}{n-3}} g_{E}$ has constant scalar curvature $\frac{n-2}{2}$, where $g_E$ is the Euclidean metric.  Equivalently,
	\ba\label{zheng}
	\Delta u+\frac{n-3}{8} u^{\frac{n+1}{n-3}} =0 \quad on \quad  \mathbb{R}^{n-1}.
	\ea
	By  Caffarelli-Gidas-Spruck \cite{Caffarelli-Gidas-Spruck} or Chen-Li \cite{Chen-Li}, the solution to (\ref{zheng}) has been classified, and none is complete.  Contradiction. The case (2) is impossible, since $N$ is noncompact. If case (3) happens,  it contradicts with $\lambda_1(g_{N^{n-1}})\geq \delta$.  Case (4) can not happen, since $N$ is nonflat.
	
	In all, the Weyl curvature of $N^{n-1}$ couldn't be zero, contradiction.
\end{proof}

\section{The Proof of Theorem \ref{main}}\label{sec6}

In this section, we can finish the proof of Theorem \ref{main}.

\noindent\textbf{Proof of Theorem \ref{main}}

By Theorem \ref{bounded curvature} the curvature is bounded, so $|\nabla Ric|$ is bounded, hence $|\frac{\nabla f\cdot\nabla(\lambda_1+\lambda_2)^2}{f}| \leq \delta(\lambda_1+\lambda_2)$, where $\delta>0$ is sufficiently small.  Then we can apply Proposition \ref{lap12} to obtain that \begin{align*}
\Delta_f(\lambda_1+\lambda_2)-\frac{2}{n-3}\nabla \log f\cdot\nabla (\lambda_1+\lambda_2)\leq -\frac{1}{n-3}(\lambda_1+\lambda_2)
\end{align*}
on $M\setminus D(a)$ for some large $a>0$. Denote $h=f-\frac{2}{n-3}\log f$ and $u= \lambda_1+\lambda_2$, the above inequality becomes
\begin{align*}
\Delta_h u\leq -\frac{1}{n-3}(\lambda_1+\lambda_2).
\end{align*}
Integrating on $M\setminus D(a)$,
\begin{align}
-\int_{\Sigma(a)}\langle\nabla u, \frac{\nabla h}{|\nabla h|}\rangle e^{-h}=\int_{M\setminus D(a)}\Delta_h u e^{-h}\leq-\frac{1}{n-3}\int_{M\setminus D(a)}(\lambda_1+\lambda_2)e^{-h}\leq 0,\label{one direction}
\end{align}	
notice that $\frac{\nabla h}{|\nabla h|}=\frac{\nabla f}{|\nabla f|}$ and the above inequality holds for almost everywhere $b\geq a$.

	Next it sufficies to prove
	\ba\label{equivalent the other direction}
	\int_{\Sigma(b)} \langle\nabla u, \frac{\nabla h}{|\nabla h|}\rangle d\sigma_{\Sigma(b)}\leq 0
	\ea
	for some $b\geq a$.
	
	\vspace{0.5cm}
	
	For this purpose, we consider the following one parameter family of diffeomorphisms,
	\begin{equation*}
		\begin{aligned}
			\begin{cases}
				&\frac{\partial F}{\partial s}=\frac{\nabla f}{|\nabla f|^2},\\
				&F(x, a)=x \in \Sigma(a).
			\end{cases}
		\end{aligned}
	\end{equation*}
	Then $\frac{\partial }{\partial s}f(F(x, s))=\langle\nabla f, \frac{\nabla f}{|\nabla f|^2} \rangle=1$,
	and the advantage of $F$ is that it maps level set of $f$ to other level set, in particular $f(F(x, s))=s$ for any $x\in \Sigma(a)$.
	
	\smallskip
	Suppose $\{x_1, x_2, x_3,\cdot\cdot\cdot, x_{n-1}\}$ are local coordinate chart of $\Sigma(a)$, on $\Sigma(s)$, let  $g(s)(\frac{\partial }{\partial x_i}, \frac{\partial}{\partial x_j}):=g(\frac{\partial F}{\partial x_i}, \frac{\partial F}{\partial x_j})$, $d\sigma_{\Sigma(s)}=\sqrt{det(g_{ij})}dx$, where $dx=dx_1\wedge dx_2\wedge dx_3\wedge\cdot\cdot\cdot\wedge dx_{n-1}$.
	Next we compute the derivatives of $d\sigma_{\Sigma(s)}$.
	\begin{align*}
		&\frac{\partial}{\partial s}d\sigma_{\Sigma(s)}=\frac{\partial}{\partial s}\sqrt{det(g_{ij})}dx\\
		=&\frac{1}{2}\cdot 2 g^{ij}\langle \nabla_{\frac{\partial F}{\partial x_i}}\frac{\partial F}{\partial s}, \frac{\partial F}{\partial x_j}\rangle d\sigma_{\Sigma(s)}\\
		=& g^{ij}\langle \nabla_{\frac{\partial F}{\partial x_i}}\frac{\nabla f}{|\nabla f|^2}, \frac{\partial F}{\partial x_j}\rangle d\sigma_{\Sigma(s)}\\
		=&\frac{1}{|\nabla f|^2}g^{ij}\nabla^2 f(\frac{\partial F}{\partial x_i}, \frac{\partial F}{\partial x_j})d\sigma_{\Sigma(s)}\\
		=&\frac{1}{|\nabla f|^2} g^{ij}\left(\frac{1}{2}g(\frac{\partial F}{\partial x_i}, \frac{\partial F}{\partial x_j})-Ric(\frac{\partial F}{\partial x_i}, \frac{\partial F}{\partial x_j})\right)d\sigma_{\Sigma(s)}\\
		=&\frac{1}{|\nabla f|^2} (\frac{n-1}{2}-\frac{n-2}{2})d\sigma_{\Sigma(s)}=\frac{1}{2s}d\sigma_{\Sigma(s)}.
	\end{align*}
	Hence, it is easy to check that
	\ban
	\frac{\partial}{\partial s}\left(\frac{1}{\sqrt{s}}d\sigma_{\Sigma(s)}\right)=\left(-\frac{1}{2}s^{-\frac{3}{2}}+\frac{1}{\sqrt{s}}\frac{1}{2s}\right) d\sigma_{\Sigma(s)}=0.
	\ean

	Next, define a function
	\ban
	I(s)=\int_{\Sigma(s)} u\cdot \frac{1}{\sqrt{s}} d\sigma_{\Sigma(s)}.
	\ean
	and then we can compute the derivative of $I(s)$ as follows:
	\ban
	I'(s)&&=\frac{\partial}{\partial s}\int_{\Sigma(s)} u\cdot \frac{1}{\sqrt{s}} d\sigma_{\Sigma(s)}\\
	&&=\int_{\Sigma(s)} \langle \nabla u, \frac{\nabla f}{|\nabla f|^2}\rangle \frac{1}{\sqrt{s}}d\sigma_{\Sigma(s)}+\int_{\Sigma(s)} u\frac{\partial}{\partial s}\left(\frac{1}{\sqrt{s}}d\sigma_{\Sigma(s)}\right)\\
	&&=\int_{\Sigma(s)} \langle \nabla u, \frac{\nabla f}{|\nabla f|^2}\rangle \frac{1}{\sqrt{s}}d\sigma_{\Sigma(s)}\\
	&&=\frac{1}{s}\int_{\Sigma(s)} \langle \nabla u, \frac{\nabla h}{|\nabla h|}\rangle d\sigma_{\Sigma(s)},
	\ean
	where $\frac{\nabla f}{\nabla f}=\frac{\nabla h}{|\nabla h|}$ and $|\nabla f|=\sqrt{s}$ were used in the last equality.
	
	\smallskip
	Moreover, since $I(s)$ tends to zero as $s\rightarrow\infty$, there exists $b>a$ such that $I'(b)\leq 0$, i.e.
	\ban
	\int_{\Sigma(b)} \langle \nabla u, \frac{\nabla h}{|\nabla h|}\rangle d\sigma_{\Sigma(b)}\leq 0;
	\ean
	this finishes the proof of (\ref{equivalent the other direction}).

Combining (\ref{one direction}) and (\ref{equivalent the other direction}), we derive that
\ban
\lambda_1+\lambda_2=0
\ean
on $M\setminus D(b)$ for some $b$. This implies that
\ban
\lambda_3=\lambda_4=\cdot\cdot\cdot=\lambda_n\equiv \frac{1}{2}
\ean
on $M\setminus D(b)$.
Hence the function
\ban
G=tr(Ric^3)-\frac{1}{2}|Ric|^2,
\ean
is $0$ on $M\setminus D(b)$.

\smallskip
Because $G$ is an analytic function, it has to be zero, we obtain that $G\equiv 0$ on $M$. Moreover, the equation $0=\Delta_f R=R-2|Ric|^2$ implies that
\ban
G=&tr(Ric^3)-|Ric|^2+\frac{1}{4}R\\
=&\sum_{i=1}^n(\lambda_i-\frac{1}{2})^2 \lambda_i=0.
\ean
Finally we get $\lambda_1=\lambda_2\equiv 0$ and $\lambda_3=\lambda_4=\cdot\cdot\cdot=\lambda_n\equiv \frac{1}{2}$ due to $Ric\geq 0$ and the continuity of $\lambda_1+\lambda_2$. This implies the Ricci curvature has constant rank $n-2$. Therefore, the soliton   is  isometric to a finite quotient of $\mathbb{R}^2 \times \mathbb{N}^{n-2}$ by \cite{FR16}, where $\mathbb{N}^{n-2}$ is an $(n-2)$-dimensional Einstein manifold. Because each level set has vanishing  Weyl curvature, hence it it isometric to $\mathbb{R}^2\times \mathbb{S}^{n-2}$.
We have completed the proof of Theorem \ref{main}.

\qed

\section{Proof of Theorem \ref{theorem2}}
In this section, we continue to consider Ricci shrinker with constant scalar curvature. Actually, we impose additional condition that the sectional curvature has upper bound.

\noindent\textbf{Proof of Theorem \ref{theorem2}}

We apply $\Delta_f$ to $\lambda_1+\lambda_2+\cdots+\lambda_{n-k}$, and the goal is to obtain $\lambda_1+\lambda_2+\cdots+\lambda_{n-k}=0$. Notice that the Ricci curvature is nonnegative, then
\begin{align*}
		&\Delta_f \left( \lambda_1 + \lambda_2 + \cdots + \lambda_{n-k} \right)\\
		&\leq \left( \lambda_1 + \lambda_2 + \cdots + \lambda_{n-k} \right) \\
		&\quad -2\bigl(K_{12}\lambda_2 + K_{13}\lambda_3 + \cdots + K_{1n}\lambda_n\bigr) \\
		&\quad -2\bigl(K_{21}\lambda_1 + K_{23}\lambda_3 + \cdots + K_{2n}\lambda_n\bigr) \\
		&\quad - \cdots - 2\bigl(K_{n-k,1}\lambda_1 + K_{n-k,2}\lambda_2 + \cdots + K_{n-k,n}\lambda_n\bigr) \\
		&= \left( \lambda_1 + \lambda_2 + \cdots + \lambda_{n-k} \right)- 2\left( K_{12}\lambda_2+\cdots+K_{1,n-k}\lambda_{n-k} \right)\\
      &\quad-2\left(K_{21}\lambda_1+\cdots+K_{2,n-k}\lambda_{n-k}  \right)\\
      &\quad-\cdots -2\left(K_{n-k, 1}\lambda_1+\cdots+K_{n-k,n-k-1}\lambda_{n-k-1}  \right)\\
		&\quad -2\lambda_{n-k+1}\left( K_{1,n-k+1} + K_{2,n-k+1} + \cdots + K_{n-k,n-k+1} \right) \\
		&\quad -2\lambda_{n-k+2}\left( K_{1,n-k+2} + K_{2,n-k+2} + \cdots + K_{n-k,n-k+2} \right) \\
		&\quad -\cdots \\
		&\quad -2\lambda_n\left( K_{1n} + K_{2n} + \cdots + K_{n-k,n} \right).
	\end{align*}

For simplicity, we use $|Rm(x)|$ to denote the norm of the Riemannian curvature at $x$.
Since the sectional curvature has upper bound $\frac{1}{2(k-1)}$, so
\begin{align*}
&\Delta_f \left( \lambda_1 + \lambda_2 + \cdots + \lambda_{n-k} \right)\\
&\leq \left( \lambda_1 + \lambda_2 + \cdots + \lambda_{n-k} \right)+ 2\left( \lambda_1 + \lambda_2 + \cdots + \lambda_{n-k} \right)\cdot C\cdot|Rm(x)|\\
		&\quad -2\lambda_{n-k+1}\left(\lambda_{n-k+1}-K_{n-k+1, n-k+2}-\cdots-K_{n-k+1, n} \right) \\
		&\quad -2\lambda_{n-k+2}\left( \lambda_{n-k+2}-K_{n-k+2, n-k+1}-\cdots-K_{n-k+2, n}\right) \\
		&\quad- \cdots \\
		&\quad -2\lambda_n\left( \lambda_n-K_{n,n-k+1}-\cdots-K_{n, n-1} \right) \\
		&= \left( \lambda_1 + \lambda_2 + \cdots + \lambda_{n-k} \right) + 2\left( \lambda_1 + \lambda_2 + \cdots + \lambda_{n-k} \right)\cdot C\cdot|Rm(x)| \\
		&\quad -2\left( \lambda_{n-k+1}^2 + \lambda_{n-k+2}^2 + \cdots + \lambda_n^2 \right) \\
		&\quad +2\lambda_{n-k+1}\left( K_{n-k+1,n-k+2} + K_{n-k+1,n-k+3} + \cdots + K_{n-k+1,n} \right) \\
		&\quad +2\lambda_{n-k+2}\left( K_{n-k+2,n-k+1} + K_{n-k+2,n-k+3} + \cdots + K_{n-k+2,n} \right) \\
		&\quad +\cdots \\
		&\quad +2\lambda_n\left( K_{n,n-k+1} + K_{n,n-k+2} + \cdots + K_{n,n-1} \right) \\
		&\leq \left( \lambda_1 + \lambda_2 + \cdots + \lambda_{n-k} \right) + 2\left( \lambda_1 + \lambda_2 + \cdots + \lambda_{n-k} \right)\cdot C\cdot|Rm(x)|\\
		&\quad - 2\left( \frac{k}{4} - \lambda_1^2 - \lambda_2^2 - \cdots - \lambda_{n-k}^2 \right) \\
		&\quad +2\lambda_{n-k+1}(k-1)\frac{1}{2(k-1)} \\
		&\quad +2\lambda_{n-k+2}(k-1)\frac{1}{2(k-1)} \\
		&\quad + \cdots + 2\lambda_n(k-1)\frac{1}{2(k-1)} \\
		&= \left( \lambda_1 + \lambda_2 + \cdots + \lambda_{n-k} \right)+ 2\left( \lambda_1 + \lambda_2 + \cdots + \lambda_{n-k} \right)\cdot C\cdot|Rm(x)|\\
		&\quad - \frac{k}{2} + 2\left( \lambda_1^2 + \lambda_2^2 + \cdots + \lambda_{n-k}^2 \right) \\
		&\quad + \left( \lambda_{n-k+1} + \lambda_{n-k+2} + \cdots + \lambda_n \right) \\
		&= \left( \lambda_1 + \lambda_2 + \cdots + \lambda_{n-k} \right)+2 \left( \lambda_1 + \lambda_2 + \cdots + \lambda_{n-k} \right)\cdot C\cdot|Rm(x)| \\
		&\quad - \frac{k}{2} + 2\left( \lambda_1^2 + \lambda_2^2 + \cdots + \lambda_{n-k}^2 \right) \\
		&\quad+\left(\frac{k}{2}-\lambda_1-\lambda_2-\cdots-\lambda_{n-k}\right) \\
		&=\left( \lambda_1 + \lambda_2 + \cdots + \lambda_{n-k} \right)\cdot2C|Rm(x)|+ 2\left( \lambda_1^2 + \lambda_2^2 + \cdots + \lambda_{n-k}^2 \right).
\end{align*}

Due to the result in \cite{FR16}, we know that $\lambda_1=\lambda_2=\cdots=\lambda_{n-k}=0$ on $f^{-1}(0)$, then we can apply the strong maximum principle to obtain that $\lambda_1=\lambda_2=\cdots=\lambda_{n-k}=0$ on $M$. Equivalently, $\lambda_{n-k+1}=\cdots=\lambda_n=\frac{1}{2}$.
Recall that on shrinking gradient Ricci soliton with constant scalar curvature, we have the following formula,
\ban
|\nabla Ric|^2=2 K_{ij}\lambda_i \lambda_j-|Ric|^2,
\ean
hence
\ban
|\nabla Ric|^2 \leq 2\cdot \frac{1}{2(k-1)}\cdot\frac{1}{2}\cdot\frac{1}{2}\cdot k(k-1)-\frac{k}{4}=0,
\ean
so $\nabla Ric=0$ on $M$. De Rham's splitting theorem implies that $(M^n, g)$ is isometric to a finite quotient of $\mathbb{R}^{n-k}\times \mathbb{N}^k$, where $\mathbb{N}^k$ is an Einstein manifold with Einstein constant $\frac{1}{2}$. Since the curvature upper bound is $\frac{1}{2(k-1)}$, it is easy to see that $\mathbb{N}^k$ has constant sectional curvature, and must be isometric to $\mathbb{S}^k$.

\vspace{0.5cm}
	\textbf{Acknowledgements}.
	The  authors would like to thank Professor Xi-Nan Ma for his constant encouragement.  The  second author is grateful to Professor Fengjiang Li and Yuanyuan Qu for helpful discussion.


\vspace{0.5cm}	
	\begin{thebibliography}{10}
		
	
		
%
		\bibitem{Chen}B. L. Chen, {\em Strong uniqueness of the Ricci flow}, J. Differential Geom. 82 (2009), no. 2, 363-382.
		\bibitem{Brendle1}S. Brendle, {\em  Rotational symmetry of self-similar solutions to the Ricci flow}, Invent. Math. 194 (2013), no. 3, 731-764.
		
		\bibitem{Brendle2}S. Brendle, {\em Rotational symmetry of Ricci solitons in higher dimensions}, J. Differential Geom. 97 (2014), no. 2, 191-214.
		\bibitem{Caffarelli-Gidas-Spruck}L. Caffarelli,  B. Gidas,    J. Spruck,   {\em Asymptotic symmetry and local behavior of semilinear elliptic equations with critical Sobolev growth}, Comm. Pure Appl. Math. 42 (1989), no. 3, 271-297.
		\bibitem{Cao}H. D. Cao, {\em Existence of gradient K\"ahler-Ricci solitons}, Elliptic and parabolic methods in geometry (Minneapolis, MN, 1994), 1-16.
		
		\bibitem{Cao-Chen}H. D. Cao, Q. Chen, {\em On locally conformally flat gradient steady Ricci solitons}, Trans. Amer. Math. Soc. 364 (2012), no. 5, 2377-2391.
		
		\bibitem{Cao-Chen2}H. D. Cao, Q. Chen,     {\em On Bach-flat gradient shrinking Ricci solitons}, Duke Math. J. 162 (2013), no. 6, 1149-1169.
		
		\bibitem{Cao-Chen-Zhu}H. D. Cao, B. L. Chen, X. P. Zhu, {\em Recent developments on Hamilton's Ricci flow}, Surveys in differential geometry. Vol. XII. Geometric flows, 47-112.
		
		\bibitem{Cao-Zhou}H. D. Cao, D. T. Zhou, {\em On complete gradient shrinking Ricci solitons}, J. Differential Geom. 85 (2010), no. 2, 175-185.
		
		\bibitem{Cao-Xie}H. D. Cao, J. M. Xie, {\em Four-dimensional complete gradient shrinking Ricci solitons with half positive isotropic curvature}, Math. Z., 305 (2023), no.2, 22 pp.
		
		\bibitem{Cao-Wang-Zhang}X. D. Cao, B. Wang, Z. Zhang, {\em On locally conformally flat gradient shrinking Ricci solitons}, Commun. Contemp. Math. 13 (2011), no. 2, 269-282.

		\bibitem{Carron-Herzlic}G. Carron, M. Herzlich, {\em Conformally flat manifolds with nonnegative Ricci curvature}
		Compos. Math. 142 (2006), no. 3, 798-810.







		
		
		\bibitem{Cheeger-Gromoll}J. Cheeger, D. Gromoll, {\em The splitting theorem for manifolds of nonnegative Ricci curvature}, J. Differential Geometry 6 (1971/72), 119-128.
		\bibitem{Chen-Li}W. X. Chen, C. M. Li,   {\em Classification of solutions of some nonlinear elliptic equations}, Duke Math. J. 63 (1991), no. 3, 615-622.
		\bibitem{Chen-Wang}X. X. Chen, Y. Q. Wang, {\em On four-dimensional anti-self-dual gradient Ricci solitons}, J. Geom. Anal. 25 (2015), no. 2, 1335-1343.
		
		\bibitem{Chen-Zhu} B. L. Chen, X.P. Zhu, {\em Uniqueness of the Ricci flow on complete noncompact manifolds}, J. Differential Geom.,74(2006), no. 1, 119-154.
		
		
		
		\bibitem{Cheng-Zhou}X. Cheng, D. T. Zhou, {\em Rigidity of Four-dimensional gradient shrinking Ricci solitons}, J. Reine Angew. Math., 802 (2023), no. 2, 255-274.
		
		
		
	
		
		\bibitem{FR16}M. Fern\'{a}ndez-L\'{o}pez, E. Garc\'{\i}a-R\'{\i}o, {\em On gradient Ricci solitons with constant scalar curvature}, Proc. Amer. Math. Soc. 144 (2016), no. 1,369-378.
		
		\bibitem{Hamilton}R. S. Hamilton, {\em Three-manifolds with positive Ricci curvature}, J. Differential Geom.  17  (1982), no. 2, 255-306.
		
		
		
		\bibitem{Kotschwar}B. Kotschwar, {\em On rotationally invariant shrinking Ricci solitons}, Pacific J. Math. 236 (2008), no. 1, 73-88.
		\bibitem{Li-Ou-Qu-Wu}F. J. Li, J. Y. Ou, Y. Y. Qu, G. Q. Wu, {\em Rigidity of Five-dimensional Shrinking Gradient Ricci Solitons with Constant Scalar Curvature}, arXiv:2411.10712.
		
\bibitem{Li-Wang} Y. Li, B. Wang, Heat kernel on Ricci shrinkers, Calc. Var. 59 (2020), Art. 194.
\bibitem{Li-Wang2}Y. Li, B. Wang, {\em  Heat kernel on Ricci shrinkers (II)}, Acta Math. Sci. Ser. B (Engl. Ed.), 44 (2024), no. 5, 1639-1695.
		
		\bibitem{Eminenti-LaNave-Mantegazza}M. Eminenti, G. LaNave, C. Mantegazza,  {\em Ricci solitons: the equation point of view}, Manuscripta Math. 127 (2008), no. 3, 345-367.
		
		
		
		
		
		
		\bibitem{Munteanu-Wang3}O. Munteanu, J. P. Wang, {\em Geometry of shrinking Ricci solitons}, Compos. Math. 151 (2015), no. 12, 2273-2300.
		
		\bibitem{Munteanu-Wang4}O. Munteanu, J. P. Wang, {\em Positively curved shrinking Ricci solitons are compact}, J. Differential Geom. 106 (2017), no. 3, 499-505.
		
	
		
		\bibitem{Naber}A. Naber, {\em  Noncompact shrinking four solitons with nonnegative curvature}, J. Reine Angew. Math. 645 (2010), 125-153.
		
		\bibitem{Ni-Wallach}L. Ni, N. Wallach, {\em On a classification of gradient shrinking solitons}, Math. Res. Lett. 15 (2008), no. 5, 941-955.
		
		\bibitem{Perelman1}G. Perelman,  {\em The entropy formula for the Ricci flow and its geometric applications}, arXiv:math/0211159v1.
		
		\bibitem{Perelman2}G. Perelman,  {\em Ricci flow with surgery on three-manifolds}, arXiv: math/0303109.
		
 	
		
		\bibitem{Petersen-Wylie}P. Petersen, W. Wylie, {\em On the classification of gradient Ricci solitons}, Geom. Topol. 14 (2010), no. 4, 2277-2300.
		
		
		\bibitem{Pigola-Rimoldi-Setti}S. Pigola, M. Rimoldi, A. Setti,   {\em Remarks on non-compact gradient Ricci solitons}, Math. Z. 268 (2011), no. 3-4, 777-790.
		\bibitem{Ou-Qu-Wu}J. Y. Ou, Y. Y. Qu, G. Q. Wu, {\em Some rigidity results on shrinking gradient Ricci soliton}, arXiv:2411.06395.
				\bibitem{Wu-Wu}G. Q. Wu, J. Y. Wu, {\em Four dimensional shrinkers with nonnegative Ricci curvature}, preprint.

	
		\bibitem{Wu-Zhang}G. Q. Wu, S. J. Zhang,{\em Remarks on shrinking gradient K\"ahler-Ricci solitons with positive bisectional curvature}, C. R. Math. Acad. Sci. Paris 354 (2016), no. 7, 713-716.
		
		\bibitem{Wu-Wu-Wylie}P. Wu, J. Y. Wu, W. Wylie, {\em Gradient shrinking Ricci solitons of half harmonic Weyl curvature}, Calc. Var. Partial Differential Equations 57 (2018), no. 5, Art. 141, 15 pp.
		
		\bibitem{Zhang}Z. H. Zhang, {\em Gradient shrinking solitons with vanishing Weyl tensor}, Pacific J. Math. 242 (2009), no. 1, 189-200.
		
		\bibitem{Zhu}S. H. Zhu, {\em The classification of complete locally conformally flat manifolds of nonnegative Ricci curvature}, Pacific J. Math. 163 (1994), no. 1, 189-199.
	\end{thebibliography}
\end{document}